\documentclass[12pt]{amsart}
\usepackage{amssymb,amsmath,amsfonts,latexsym,amscd}
\usepackage{bm}
\newcommand{\newsection}[1]{\setcounter{equation}{0} \section{#1}}
\theoremstyle{plain}
\newtheorem{propn}{Proposition}[section]

\newtheorem{cor}[propn]{Corollary}
\newtheorem*{thm*}{Theorem}

\theoremstyle{definition}

\newtheorem*{note}{Note}

\theoremstyle{remark}
\newtheorem*{rem}{Remark}

\newcommand{\Hil}{\mathcal{H}}
\newcommand{\q}{\mathcal{Q}}
\newcommand{\s}{\mathcal{S}}

\newcommand{\clb}{\mathcal{B}}

\newcommand{\cle}{\mathcal{E}}

\newcommand{\clh}{\mathcal{H}}

\newcommand{\clq}{\mathcal{Q}}

\newcommand{\cls}{\mathcal{S}}

\newcommand{\z}{\bm{z}}
\newcommand{\w}{\bm{w}}

\newcommand{\I}{\mathcal{I}}
\newcommand{\Nat}{\mathbb{N}}
\newcommand{\Comp}{\mathbb{C}}
\newcommand{\Int}{\mathbb{Z}}
 \newcommand{\D}{\mathbb{D}}
\newcommand{\ot}{\otimes}

\newcommand{\vp}{\varphi}

\newtheorem{Theorem}{\sc Theorem}[section]
\newtheorem{Lemma}[Theorem]{\sc Lemma}
\newtheorem{Proposition}[Theorem]{\sc Proposition}
\newtheorem{Corollary}[Theorem]{\sc Corollary}
\newtheorem{Definition}[Theorem]{\sc Definition}
\newtheorem{Example}[Theorem]{\sc Example}
\newtheorem{Remark}[Theorem]{\sc Remark}
\newtheorem{Note}[Theorem]{\sc Note}
\newtheorem{Question}{\sc Question}
\newtheorem{ass}[Theorem]{\sc Assumption}
\newcommand{\bt}{\begin{Theorem}}
\def\beginlem{\begin{Lemma}}
\def\beginprop{\begin{Proposition}}
\def\begincor{\begin{Corollary}}
\def\begindef{\begin{Definition}}
\def\beginexamp{\begin{Example}}
\def\beginrem{\begin{Remark}}
\def\beginq{\begin{Question}}
\def\beginass{\begin{ass}}
\def\beginnote{\begin{Note}}
\newcommand{\et}{\end{Theorem}}
\def\endlem{\end{Lemma}}
\def\endprop{\end{Proposition}}
\def\endcor{\end{Corollary}}
\def\enddef{\end{Definition}}
\def\endexamp{\end{Example}}
\def\endrem{\end{Remark}}
\def\endq{\end{Question}}
\def\endass{\end{ass}}
\def\endnote{\end{Note}}

\begin{document}

\title{Star-Generating Vectors of Rudin's Quotient Modules}

\author[Chattopadhyay] {Arup Chattopadhyay}
\address{%\vskip10pt
(A. Chattopadhyay) Indian Statistical Institute \\ Statistics
andMathematics Unit \\ 8th Mile, Mysore Road \\ Bangalore \\ 560059
\\ India}

\email{2003arupchattopadhyay@gmail.com, arup@isibang.ac.in}

\author[Das] {B. Krishna Das}

\address{%\vskip10pt
(B. K. Das) Indian Statistical Institute \\ Statistics and
Mathematics Unit \\ 8th Mile, Mysore Road \\ Bangalore \\ 560059 \\
India} \email{dasb@isibang.ac.in, bata436@gmail.com}

\author[Sarkar]{Jaydeb Sarkar}

\address{%\vskip10pt
(J. Sarkar) Indian Statistical Institute \\ Statistics and
Mathematics Unit \\ 8th Mile, Mysore Road \\ Bangalore \\ 560059
\\ India}

\email{jay@isibang.ac.in, jaydeb@gmail.com}


\begin{abstract}
The purpose of this paper is to study a class of quotient modules of
the Hardy module $H^2(\mathbb{D}^n)$. Along with the two variables
quotient modules introduced by W. Rudin, we introduce and study a
large class of quotient modules, namely Rudin's quotient modules of
$H^2(\mathbb{D}^n)$. By exploiting the structure of minimal
representations we obtain an explicit co-rank formula for Rudin's
quotient modules.
\end{abstract}

\maketitle

\section*{Notation}

\begin{list}{\quad}{}
\item $\mathbb{N}$ \quad \quad \; Set of all natural numbers
including 0.

\item $n$ \quad \quad \; Natural number $n \geq 2$, unless
specifically stated otherwise.
\item $\mathbb{N}^n$ \quad \quad $\{\bm{k} = (k_1, \ldots, k_n) : k_i \in \mathbb{N}, i = 1,
\ldots, n\}$.
\item $\mathbb{C}^n$  \quad \quad Complex $n$-space.
\item $\bm{z}$ \quad \quad \; $(z_1, \ldots, z_n) \in \mathbb{C}^n$.
\item $\bm{z}^{\bm{k}}$ \quad \quad \,$z_1^{k_1}\cdots
z_n^{k_n}$.
\item $T$ \quad \quad \; $n$-tuple of commuting operators $(T_1, \ldots, T_n)$.
\item $T^{\bm{k}}$ \quad \quad $T_1^{k_1} \cdots
T_n^{k_n}$.
\item $\mathbb{C}[\z]$  \quad \, $\mathbb{C}[z_1, \ldots, z_n]$,
the polynomial ring over $\mathbb{C}$ in $n$-commuting
variables.
\item $\mathbb{D}^n$  \quad \quad Open unit polydisc $\{\z : |z_i|
<1\}$.
\end{list}

Throughout this note all Hilbert spaces are over the complex field
and separable. Also for a closed subspace $\cls$ of a Hilbert space
$\clh$, we denote by $P_{\cls}$ the orthogonal projection of $\clh$
onto $\cls$.

\newsection{Introduction}\label{sec:1}

This paper is concerned with the question of generating vectors for a
commuting tuple of bounded linear operators on a separable Hilbert
space. Let $T:= (T_1, \ldots, T_n)$ be an $n$-tuple ($n \geq 1$) of
commuting bounded linear operators on a separable Hilbert space
$\clh$, and let $S$ be a non-empty subset of $\clh$. The
$T$-generating hull of $S$ is defined by \[[S]_T: = \bigvee \{p(T_1,
\ldots, T_n) h : p \in \mathbb{C}[\z], h \in S\},\]where
$\mathbb{C}[\z] := \mathbb{C}[z_1, \ldots, z_n]$ is the ring of
polynomials of $n$-commuting variables. It is easy to verify that
$[S]_T$ is the smallest closed subspace of $\clh$ such that $S
\subseteq [S]_T$ and $T_i ([S]_T )\subseteq [S]_T$ for all
$i=1,\dots,n$. That is,
\[[S]_T = \bigcap \{\cls : \cls
\mbox{~closed}, S \subseteq \cls, T_i \cls \subseteq \cls, ~i = 1,
\ldots, n\}.\] The \textit{rank} of the tuple $T$ is the unique
number $\mbox{rank~} T$ defined by \[ \mbox{rank~} T =
\mbox{inf}\{\# S : S \subseteq \clh, ~[S]_T = \clh\}\in \mathbb{N} \cup \{\infty\}.\]
We say that $S$ is
a $T$-generating subset if $[S]_T = \clh$. In this case, we also say
that the elements in $S$ are \textit{$T$-generating vectors} of
$\clh$.

\noindent One of the most intriguing and important open problems in
operator theory, closely related to the invariant subspace problem,
is the existence of nontrivial generating set for a tuple of operators.
Also one may ask when the rank of $T$ is finite.

This problem is known to be hard to solve in general. Nevertheless,
one has a better chance in the special case of the restriction (compression)
of tuple of multiplication operators, $M_z = (M_{z_1},
\ldots, M_{z_n})$, to a joint $M_z$-invariant (co-invariant)
subspace of an analytic reproducing kernel Hilbert space $\clh$ over
a domain in $\mathbb{C}^n$. The reason behind this hope is that
underlying analytic structure of the kernel function could
provide useful information regarding the restriction (compression)
of $M_z$ to a joint $M_z$-invariant (co-invariant) subspace of
$\clh$.

In this paper, we confine our attention to a class of quotient
Hilbert modules, namely ``Rudin's quotient modules'', after W. Rudin
(\cite{rudin}), of the Hardy module $H^2(\D^n)$ over the unit
polydisc $\mathbb{D}^n$.  We focus here mainly on the study of
$(M_{z_1}^*|_{\clq}, \ldots, M_{z_n}^*|_{\clq})$-generating sets of
a Rudin's quotient module $\q$ of $H^2(\D^n)$. To be more precise we
must introduce some notations first.

Let $n \geq 1$ be a natural number and $\mathbb{D}^n$ be the open
unit polydisc in $\mathbb{C}^n$. The \textit{Hardy module}
$H^2(\mathbb{D}^n)$ over $\mathbb{D}^n$ is the Hilbert space of all
holomorphic functions $f$ on $\mathbb{D}^n$ such that
\[\|f\|_{H^2(\mathbb{D}^n)}:= \left(\sup_{0\leq r< 1}
\int_{\mathbb{T}^n}|f(r\z)|^2~d\z \right)^{\frac{1}{2}}< \infty,
\]where $d\z$ is the normalized Lebesgue measure on the torus
$\mathbb{T}^n$, the distinguished boundary of $\mathbb{D}^n$, and
$r\z:=\left( rz_1, \ldots,rz_n\right)$ (cf. \cite{rudin}). For each
$i = 1, \ldots, n$, define the multiplication operator by the
coordinate function $z_i$ as \[(M_{z_i} f)(\w) = w_i f(\w). \quad
\quad (\w \in \D^n)\] We will often identify $H^2(\D^n)$ with the
$n$-fold Hilbert space tensor product $H^2(\D) \otimes \cdots \otimes H^2(\D)$.
In this identification,
the multiplication operator $M_{z_i}$ can be realized as
$I_{H^2(\D)} \otimes \cdots \otimes
\underbrace{M_z}\limits_{\textup{i-th place}} \otimes \cdots \otimes
I_{H^2(\D)}$ for all $i = 1, \ldots, n$.

\noindent One can easily verify that \[M_{z_i} M_{z_j} = M_{z_j}
M_{z_i}, \quad M_{z_i}^* M_{z_i} = I_{H^2(\D^n)}, \quad \quad \quad
(i, j = 1, \ldots, n)\]that is, $(M_{z_1}, \ldots, M_{z_n})$ is an
$n$-tuple of commuting isometries. Moreover, for $n \geq 2$,
\[M_{z_i}^* M_{z_j} = M_{z_j} M_{z_i}^*, \quad \quad (1 \leq i < j
\leq n)\]that is, $(M_{z_1}, \ldots, M_{z_n})$ is a \textit{doubly
commuting} tuple of isometries.

A closed subspace $\mathcal{S}\subseteq H^2(\mathbb{D}^n)$ is said to be a
\textit{submodule} of $H^2(\mathbb{D}^n)$ if
$M_{z_i}(\mathcal{S})\subseteq \mathcal{S}$ for all
$i=1,2,\ldots,n$, and a closed subspace $\mathcal{Q}\subseteq
H^2(\mathbb{D}^n)$ is said to be a \textit{quotient module} if
$\mathcal{Q}^{\perp} \left(= H^2(\mathbb{D}^n)\ominus \mathcal{Q}
\cong H^2(\D^n)/ \clq \right)$ is a submodule of
$H^2(\mathbb{D}^n)$. For notational simplicity we shall let
\[\mbox{rank~} \cls : = \mbox{rank~}M_{\z}|_{\cls}, \quad \mbox{and}
\quad \mbox{co-rank~} \clq = \mbox{rank~} M_{\z}^*|_{\clq},\]where
$M_{\z}|_{\s} = (M_{z_1}|_{\s}, \ldots, M_{z_n}|_{\s})$ and
$M^*_{\z}|_{\q} = (M^*_{z_1}|_{\q}, \ldots, M^*_{z_n}|_{\q})$. It is
quite natural to call the vectors of a
$M_{\z}^*|_{\clq}$-generating set of a quotient module $\clq$ as
\textit{star-generating} vectors of $\q$. We say that a quotient
module $\q$ is \textit{star-cyclic} if $\mbox{co-rank~}\q = 1$.

Let us examine now the problem of generating vectors for submodules
and quotient modules of $H^2(\D)$. Let $\clq$ be a
quotient module of $H^2(\D)$. Then from Beurling's theorem it
follows that $\clq = \clq_{\vp}$ for some inner function $\vp \in
H^{\infty}(\D)$ (that is, $\vp$ is a bounded holomorphic function on
$\D$ and $|\vp| = 1$ a.e on $\mathbb{T}$), where
\[\clq_{\vp} := H^2(\D) \ominus \cls_{\varphi},
\;\mbox{and} \; \cls_{\varphi} := \varphi H^2(\D).\]
Next we consider the following
two cases:

\noindent Case I: For the submodule $\cls_{\varphi}$, it
follows that $\varphi$ is a generating vector for $M_z|_{\cls_{\varphi}}$. In
particular, $\cls_{\vp}$ is singly generated.

\noindent Case II: For the quotient module $\clq_{\vp}$, it turns out that
$M_z^* \varphi$ is a generating vector for $M_z^*|_{\clq_{\vp}}$ (see
Proposition \ref{*-one}). In particular, $\clq_{\vp}$ is singly
star-generated.

The above conclusions fail if we replace $H^2(\D)$ by $H^2(\D^2)$.
The following counterexample, due to Rudin \cite{rudin}, is
particularly concise and illustrate the heuristic ideas behind our
general consideration: Let $\s$ be a submodule of $H^2(\D^2)$
consists of all functions in $H^2(\D^2)$ which have a zero of order
greater than or equal to $m$ at $(0,\alpha_m)=(0, 1-m^{-3})$. Then
$\mbox{rank~} {\s} = \infty$. Without giving a proof we note that
the above submodule can be represented in the following way (see
\cite{Se}):
\[
\s=\s_{\Phi}:= \bigvee_{m=0}^{\infty} \varphi_m H^2(\D) \otimes z^m
H^2(\mathbb{D}),\]
 where
$\Phi=\{\varphi_m\}_{m\ge 0}$ is the decreasing sequence of Blaschke
products defined by
\[
\varphi_0=\prod_{i=1}^{\infty}b_{\alpha_i}^i,\quad\quad \varphi_m =
\frac{\varphi_{m-1}} {\prod_{i=m}^{\infty}b_{\alpha_i}}.\quad
\quad\quad (m\ge 1)
\]Here and throughout this paper, for each $\alpha \in \D$, we denote
by $b_{\alpha}$ the Blaschke factor
\[b_{\alpha}(z) = \frac{z - \alpha}{1 - \bar{\alpha} z}. \quad \quad
\quad (z \in \D)\]

Subsequently, Rudin's result was improved and analyzed for submodules
and quotient modules of $H^2(\D^2)$ by several authors, such as
Ahern and Clark \cite{AC}, Agrawal, Clark and Douglas \cite{ADC}, K.
J. Izuchi, K. H. Izuchi and Y. Izuchi \cite{I1}, \cite{I2}, Seto and
Yang \cite{Se}.

Inspired and motivated by the above fact, we introduce the notion
of a Rudin's quotient module: Let $\Phi_i = \{\varphi_{i,k}\}_{k=
-\infty}^{\infty}$ be a sequence of Blaschke products (see
\eqref{Blas-def}) for all $i = 1, \ldots, n$. Moreover, we assume that
for each $i = 1, \ldots, n$, $\{\varphi_{i,k}\}_{k=
-\infty}^{\infty}$ has a least common multiple $\vp_i$,
that is, $\varphi_{i,k} |\varphi_i$ for all $-\infty<k<\infty$ and
if $\psi_i$ is a Blaschke product with the same property and
$\psi_i|\varphi_i$, then $\psi_i$ is a constant multiple of
$\varphi_i$. The \textit{Rudin's quotient module}
corresponding to $\bm{\Phi} := (\Phi_1, \ldots, \Phi_n)$ is the
quotient module $\q_{\bm{\Phi}}$ of $H^2(\D^n)$ defined by
\[\begin{split} \q_{\bm{\Phi}} & =
\bigvee_{k=-\infty}^{\infty}\left(H^2(\D) \ominus \varphi_{1,k}
H^2(\D)\right) \otimes \cdots \otimes \left(H^2(\D)\ominus
\varphi_{n,k} H^2(\D)\right)\\ & =  \bigvee_{k=-\infty}^{\infty}
\clq_{\varphi_{1,k}} \otimes \cdots \otimes \clq_{\varphi_{n,k}}.
\end{split}
\]As we have already mentioned, when $n = 2$, in various situations,
such Rudin's quotient module were already considered by Rudin
\cite{rudin}, K. J. Izuchi et al. \cite{I1, I2}, Young and Seto
\cite{Se} and Seto \cite{MS} (for $n \geq 2$, see Douglas et al.
\cite{DPSY}, Guo \cite{Guo1, Guo2}, Sarkar \cite{JS1, JS2} and
Chattopadhyay et al. \cite{CDS}).

The purpose of this paper is to compute and analyze the co-rank of
$\clq_{\bm{\Phi}}$. We also consider the case when some of the inner
sequences are increasing sequence of Blaschke products and rest of
them are decreasing sequence of Blaschke products.  Along the way,
we obtain some results concerning minimal representations and
compute co-ranks of a class of quotient modules of
$H^2(\mathbb{D}^n)$.

Some of our main results are generalizations of theorems due to K.
J. Izuchi, K. H. Izuchi and Y. Izuchi \cite{I1}. One of our main
results, Theorem \ref{mainth2}, concerning co-rank of a Rudin's
quotient module is a refined and generalized version of results by
Izuchi et al. (Theorems 4.2 and 4.3 in \cite{I1}). In particular, we
point out and correct an error in the proof of the main result in
the paper by Izuchi et al. (see Remark \ref{ce}).

The remainder of the paper is organized as follows. In the
preliminary Section 2, we set up notations, definitions and results
needed further. In Section 3, we introduce and investigate minimal
representations of a class of finite dimensional quotient modules of
$H^2(\mathbb{D}^n)$. Section 4 is devoted to the minimal
representation of a zero-based quotient module of
$H^2(\mathbb{D}^n)$. In Section 5, the theory of Sections 3 and 4 is
applied to obtain the main results of this paper. In the last
section, we give some (counter-) examples of Rudin's quotient
module of $H^2(\mathbb{D}^2)$.

\newsection{Preliminaries and preparatory results}

In this section we gather some facts concerning quotient modules of
$H^2(\D)$. We begin with the result that any quotient module of
$H^2(\D)$ is singly generated.

\begin{Proposition}\label{*-one}
For an inner function $\vp$, let $\clq_{\vp}=H^2(\D)\ominus \vp H^2(\D)$
be a quotient module of $H^2(\D)$. Then $M_z^* \vp$
is a star-cyclic vector of $\clq_{\vp}$.
\end{Proposition}
\begin{proof}
Since $M_{\vp}^* M_z^* = M_z^* M_{\vp}^*$, for all $m \geq 0$ we
have
\[\langle M_z^* \vp, \vp z^m \rangle = \langle M_{\vp}^* M_z^* \vp,
z^m \rangle = \langle M_z^* M_{\vp}^* M_{\vp} 1, z^m \rangle =
\langle M_z^* 1, z^m \rangle = 0.\] From this it follows that $M_z^*
\vp \in \clq_{\vp}$. Then there exists an inner function $\theta$
such that  \[\clq_{\theta} = \mathop{\bigvee}_{m
\geq 0} M_z^{*m} (M_z^* \vp) \subseteq \clq_{\vp}.\]
Now if the above inclusion is proper, then
 $\vp = \theta \psi$
for some non-constant inner function $\psi$. On the other hand,
$\theta \perp \clq_{\theta} = \mathop{\bigvee}_{m \geq 0} M_z^{*m}
(M_z^* \vp)$ implies that
\[0 = \langle M_z^{*m} \vp, \theta \rangle = \langle M_z^{*m} \theta \psi, \theta \rangle
= \langle \psi, z^m \rangle,\]for all $m \geq 1$, which is a contradiction as  $\psi$ is a
non-constant inner function. Therefore $\q_{\theta} = \q_{\vp}$ and the proof follows.
\end{proof}

The following well-known result furnish a rich supply of star-cyclic
vectors for a quotient module of $H^2(\D)$.

\begin{Proposition}\label{mainlemma2}
Let $\varphi$ and $\psi$ be two non-constant inner functions and $q$
be the greatest common inner factor of $\varphi$ and $\psi$. Let
$\theta=\varphi/q$ and $f$ be a star-cyclic vector of
$\q_{\varphi}$. Then $M_{\psi}^*f$ is also a star-cyclic vector of
$\q_{\theta}$.
\end{Proposition}

\begin{proof}
Since $f\in\q_{\varphi}$, we have
\[
 \langle M_{\psi}^*f, \theta z^m\rangle=\langle f, \varphi (\psi/q)z^m\rangle
 =0. \quad \quad (m \ge 0)
\]
Thus it follows that $M_{\psi}^*f\in \q_{\theta}$ and
$\bigvee\limits_{m=0}^{\infty}M_z^{*m}M_{\psi}^*f\subseteq
\q_{\theta}$. We need to show $ \q_{\theta}\subseteq
\bigvee\limits_{m=0}^{\infty}M_z^{*m}M_{\psi}^*f$. Now for $g \in
\q_{\theta}$ and $g \perp
\bigvee\limits_{m=0}^{\infty}M_z^{*m}M_{\psi}^*f$, we have
\[\psi g\perp \bigvee\limits_{m=0}^{\infty}M_z^{*m}f=\q_{\varphi},\]
and therefore $\psi g\in \psi H^2(\D)\cap \theta H^2(\D)=\psi \theta
H^2(\D)$. Consequently $\psi g= \psi \theta h$ for some $h\in H^2(\D)$.
Thus
$g= \theta h$ and together with $g\in\q_{\theta}$ imply that $g=0$.
The proof is complete.
\end{proof}

%To state the next result, we need to recall the notion of relatively
%prime inner functions.

A pair of non-constant inner functions $\varphi$ and $\psi$ is said
to be \emph{relatively prime} if $\vp$ and $\psi$ do not have any
common non-constant inner factor. An immediate consequence of
the above proposition is as follows:

\begin{Corollary}\label{mainlemma1}
Let $\vp$ and $\psi$ be two relatively prime inner functions and $f$
be a star-cyclic vector of $\clq_{\vp}$. Then $M_{\psi}^*f$ is also
a star-cyclic vector of $\q_{\varphi}$.
\end{Corollary}

We now specialize to the case where $\varphi$ is a \textit{Blaschke
product}, that is,
\begin{equation}\label{Blas-def}\varphi=\prod_{m
=1}^{\infty}b_{\alpha_m}^{l_m},\end{equation}where
$\{l_m\}_{m=1}^{\infty}$ is a sequence of natural numbers and
$\{\alpha_m\}_{m =1}^{\infty} \subseteq \D$ is a sequence of
distinct scalars satisfying \[\sum_{m = 1}^{\infty}(1 -
l_m|\alpha_m|)<\infty.\]

Let $\varphi \in H^\infty(\D)$ be a Blaschke product, and let
${\mathcal{I}}_{\varphi}$ denote the set of all relatively prime
inner factors of $\varphi$ such that
\begin{equation}\label{relatively prime}
\varphi=\prod_{\xi\in \I_{\varphi}}\xi.
\end{equation} Note that $\I_{\varphi}$ is a
countable set and contains non-constant inner functions. Moreover,
for each $\xi \in \I_{\vp}$, there exists a unique prime inner
function $P(\xi)$ and an integer $m \in \mathbb{N}\setminus \{0\}$ such that
\[\xi = P(\xi)^m.\] In particular if $\vp$ is of the form ~\eqref{Blas-def}, then
\[\I_{\varphi}=\{b_{\alpha_m}^{l_m}: m \ge 1\},\] and
$P(b_{\alpha_m}^{l_m})= b_{\alpha_m}$ for all $m \ge 1$.

The following result relates the aspect of relatively prime factors
of a given inner function $\vp$ to the corresponding quotient module
$\q_{\vp}$.

\begin{Lemma}
Let $\varphi$ be a Blaschke product. Then
\[
\q_{\varphi}= \bigvee_{\xi\in\I_{\varphi}}\q_{\xi}.
\]
\end{Lemma}
\begin{proof} Since $\varphi = \prod\limits_{\xi\in\I_{\varphi}}\xi$,
it follows that $\varphi H^2(\D)\subseteq \xi H^2(\D)$ for all
$\xi\in \I_{\varphi}$. This implies that $\q_{\xi}\subseteq
\q_{\varphi}$ for all $\xi\in\I_{\varphi}$, and therefore
$\bigvee_{\xi\in\I_{\varphi}}\q_{\xi} \subseteq \q_{\varphi}$.

\noindent We now proceed to prove the other inclusion. Let $f\in
\left(\bigvee_{\xi\in\I_{\varphi}}\q_{\xi}\right)^{\perp}$. Then $
f\in \q_{\xi}^{\perp}=\xi H^2(\D)$ for all $\xi\in\I_{\varphi}$,
that is, $f\in \mathop{\bigcap}_{\xi\in\I_{\varphi}}\xi H^2(\D)$.
Since $\I_{\varphi}$ contains only relatively prime Blaschke
products, it follows that
\[f\in \mathop{\bigcap}_{\xi\in\I_{\varphi}}\xi H^2(\D) =
\big(\prod\limits_{\xi\in\I_{\varphi}}\xi \big) H^2(\D) = \varphi
H^2(\D)=\q_{\varphi}^{\perp}.
\]
This completes the proof.
\end{proof}

As a corollary, we obtain the following useful fact for tensor
product of quotient modules:

\begin{Corollary} \label{maincor3}
Let $\{\varphi_j\}_{j=1}^n$ be a collection of Blaschke products.
Then
\[
\q_{\varphi_1}\otimes \cdots \otimes \q_{\varphi_n}
=\bigvee_{(\xi_1,\dots,\xi_n)\in
\I_{\varphi_1}\times\cdots\times\I_{\varphi_n}} \q_{\xi_{1}}\otimes
 \cdots \otimes \q_{\xi_{n}}.
\]
\end{Corollary}

The following lemmas are both simple and useful.

\begin{Lemma}
\label{lemma 1} Let $\{\xi_i\}_{i=1}^n$ and $\{\eta_i\}_{i=1}^n$ be
inner functions such that $\xi_j|\eta_j$ for some $1\le j \le n$.
Then
 \[
M_{\eta_1}^*\ot\cdots\ot
M_{\eta_n}^*(\q_{\xi_1}\ot\cdots\ot\q_{\xi_n})=\{0\}.
 \]
\end{Lemma}
\begin{proof}
Let $1\le j \le n$ be such that $\xi_j|\eta_j$. Since
$\cls_{\xi_j}\supseteq \cls_{\eta_j}$, we have $\q_{\xi_j}\subseteq
\q_{\eta_j}$. Then the proof follows from the fact that
$M_{\eta_j}^*(\q_{\eta_j})=\{0\}$.
\end{proof}

\begin{Lemma}\label{lemma 2}
Let $T=(T_1,\dots,T_n)$ be a commuting tuple of operators on a
Hilbert space $\Hil$, and let $\q$ be a joint $T^*$-invariant closed
subspace of $\Hil$. Then
\[\mbox{rank~}P_{\q}T\vert_{\q} \leq \mbox{rank~}T,\]
where $P_{\q}T\vert_{\q} : = (P_{\q}T_1 \vert_{\q}, \ldots,
P_{\q}T_n \vert_{\q})$.
\end{Lemma}

\begin{proof}
If $\mbox{rank~} T=\infty$, then there is nothing to prove. So, let
$\{f_1,\dots,f_m\} \subseteq \clh$ be a $T$-generating set for some
$m \in \mathbb{N}$. Since $P_{\q} T_j P_{\q} = P_{\q} T_j$ for all
$j = 1, \ldots, n$, we have \[(P_{\q} T P_{\q})^{\bm{k}} (P_{\q}
f_l) = P_{\q} T^{\bm{k}} P_{\q} f_l = P_{\q} (T^{\bm{k}} f_l),\]for
all $l = 1, \ldots, m$ and ${\bm{k}} \in \mathbb{N}^n$. On the other
hand, since $\bigvee \{ T^{\bm{k}} f_j : {\bm{k}} \in \mathbb{N}^n,
j = 1, \ldots, m\} = \clh$, we have $\bigvee \{ (P_{\q} T
P_{\q})^{\bm{k}} (P_{\q} f_j) : {\bm{k}} \in \mathbb{N}^n, j = 1,
\ldots, m\} = \q$. This shows in particular that
$\{P_{\q}f_1,\dots,P_{\q}f_m\}$ is a
$P_{\q}T\vert_{\q}$-generating subset of $\q$. This completes the
proof.
\end{proof}

\newsection{Co-ranks of finite dimensional quotient modules}\label{sec:2}

In this section we determine co-ranks of some finite dimensional
quotient modules of $H^2(\D^n)$. This will be particularly useful in
the next section when we consider minimal representations of
quotient modules.

Before proceeding further, we find more useful descriptions of
finite dimensional quotient modules of $H^2(\D^n)$. Recall that a
quotient module $\q_{\vp}$ is finite dimensional if and only if
$\vp$ is a finite Blaschke product, which is unique up to the circle
group $\mathbb{T}$, and $\mbox{order} \vp = \mbox{dim~} \clq_{\vp}$.
Here our main interest concern the case of $\vp = b_{\alpha}^m$,
where $\alpha \in \D$ and $m \in \mathbb{N}$. We first observe that
for $\alpha \in \D$ and $m \geq 1$, $\{b_{\alpha}^j M_z^*
b_{\alpha}\}_{j=0}^{m-1}$ is an orthogonal basis of the quotient
module $\clq_{b_{\alpha}^m}$. A simple calculation reveals that
\[M_z^* b_{\alpha} = (1 - |\alpha|^2) \mathbb{S}(\cdot,
\alpha),\]where $\mathbb{S}(\cdot, \alpha)$ is the Szeg\"{o} kernel
on $\D$ defined by \[\mathbb{S}(\cdot, \alpha)(z) = (1 -
\bar{\alpha} z)^{-1}. \quad \quad (z \in \D)\] Since
$M_{b_{\alpha}^j} \in \clb(H^2(\D))$ is an isometry, we have
\begin{equation}\label{norm-1}\|b_{\alpha}^j M_z^* b_{\alpha}\| =
\|M_z^* b_{\alpha}\| = (1 - |\alpha|^2) \|\mathbb{S}(\cdot,
\alpha)\| = (1 - |\alpha|^2)^{\frac{1}{2}},\end{equation}for all $j
\in \mathbb{N}$. Obviously \[\langle b_{\alpha}^{m-1} M_z^*
b_{\alpha}, M_z^*(b_{\alpha}^{m-1} M_z^* b_{\alpha}) \rangle =
\langle M_z^* b_{\alpha}, M_z^{*2} b_{\alpha} \rangle = \bar{\alpha}
(1-|\alpha|^2)^2 \langle \mathbb{S}(\cdot, \alpha), \mathbb{S}(\cdot,
\alpha) \rangle = \bar{\alpha}(1-|\alpha|^2),\] which yields
\begin{equation}\label{P-1}
P_{\mathbb{C} (b_{\alpha}^{m-1} M_z^* b_{\alpha})} M_z^*
b_{\alpha}^{m-1} M_z^* b_{\alpha} = \bar{\alpha} (b_{\alpha}^{m-1}
M_z^* b_{\alpha}),
\end{equation}
where $m \geq 1$ and $P_{\mathbb{C}(b_{\alpha}^{m-1} M_z^*
b_{\alpha})}$ denotes the orthogonal projection of $H^2(\D)$ onto the
one dimensional subspace generated by the vector $b_{\alpha}^{m-1}
M_z^* b_{\alpha}$.

We next introduce a new class of quotient modules which is based on
submodules vanishing at a point of $\mathbb{D}^n$. Given
$\bm{\alpha} = (\alpha_1,\dots,\alpha_n)\in \D^n$ and a finite
subset $A$ of $(\Nat\setminus\{0\})^n$, let $\clq(\bm{\alpha}; A)$ be the quotient
module defined by
\begin{equation}\label{nset}
\q(\bm{\alpha}; A):= \bigvee_{(l_1,\dots,l_n)\in
A}\q_{b_{\alpha_1}^{l_{1}}}\otimes \cdots \otimes
\q_{b_{\alpha_n}^{l_{n}}}.
\end{equation}
Now we want to find a minimum cardinality subset $\tilde{A}$ of $A$
such that $\q(\bm{\alpha}; A)=\q(\bm{\alpha};\tilde{A})$. In order
to find $\tilde{A}$, first observe that for $(l_1,\dots,l_n),
(l_1',\dots,l_n')\in A$ if $l_{i}\le l_{i}'$ for all $i=1,\ldots,n$,
then
\begin{equation}\label{inc-red}\q_{b_{\alpha_1}^{l_{1}}}\otimes \cdots \otimes
\q_{b_{\alpha_n}^{l_{n}}} \subseteq
\q_{b_{\alpha_1}^{l_{1}'}}\otimes \cdots \otimes
\q_{b_{\alpha_n}^{l_{n}'}},\end{equation} which implies that
\begin{equation}\label{inc-red2}
(\q_{b_{\alpha_1}^{l_1}}\ot\cdots\ot\q_{b_{\alpha_n}^{l_n}}) \vee
(\q_{b_{\alpha_1}^{l_{1}'}}\otimes \cdots \otimes
\q_{b_{\alpha_n}^{l_{n}'}}) = \q_{b_{\alpha_1}^{l_{1}'}}\otimes
\cdots \otimes \q_{b_{\alpha_n}^{l_{n}'}}.\end{equation}
Thus
$\q(\bm{\alpha}; A) = \q(\bm{\alpha}; A \setminus \{(l_1, \ldots,
l_n)\})$, and hence we remove $(l_1,\dots,l_n)$ from $A$. By
continuing this process we eventually obtain a set $\tilde{A}
\subseteq A$ of minimal cardinality such that
\begin{equation}\label{mnset}
\q(\bm{\alpha}; A) = \q(\bm{\alpha}; \tilde{A}) =
\bigvee_{(l_1,\dots,l_n)\in
\tilde{A}}\q_{b_{\alpha_1}^{l_{1}}}\otimes \cdots \otimes
\q_{b_{\alpha_n}^{l_{n}}}.
\end{equation}

It should be noted that for any pair $(l_1,\dots,l_n)$ and
$(l_1',\dots,l_n')$ of $\tilde{A}$, one has the following:
\begin{equation}
\label{equivalent definition} \exists\  i, j\in \{1,\dots,n\}, i\neq
j, \text{ such that } l_i<l_i' \text{ and }
 l_j>l_j' .
\end{equation}
The new representation $\q(\bm{\alpha}; \tilde{A})$ is called the
\textit{minimal representation} of $\q(\bm{\alpha}; A)$.

We remark here that any finite dimensional quotient module of
$H^2(\D^n)$ is span closure of finite number of quotient modules of
the form $\q(\bm{\alpha};\tilde{A})$ (cf. \cite{AC}, also see
Douglas, Paulsen, Sah and Yan \cite{DPSY}, Guo \cite{Guo1, Guo2} and
Chen and Guo \cite{CG}).

This minimal representation of $\q(\bm{\alpha}; {A})$ plays a
fundamental role in calculating the co-rank of a Rudin's quotient
module $\q$. Here is one example.

\begin{Proposition} \label{mainlemma4}
Let $\q(\bm{\alpha}; A)$ be a quotient module as in (\ref{nset}).
Then
\[\mbox{co-rank}~\q(\bm{\alpha}; A) =
\mbox{co-rank~}\q(\bm{\alpha}; \tilde{A}) = \#\tilde{A}.\]
\end{Proposition}
\begin{proof} Let $\#\tilde{A} = r$. Without loss of
generality we assume that $\tilde{A} =
\{(l_{1,k},l_{2,k},\dots,l_{n,k})\in\Nat^n: k=1,\dots, r\}$. Let
$f_{j,k}$ be a star-generator of $\q_{b_{\alpha_j}^{l_{j,k}}}$,
where $k=1,\dots, r$, and $j = 1, \ldots, n$. Then
\[
[f_{1,k}\otimes \cdots \otimes f_{n,k}]_{M_z^*} =
\q_{b_{\alpha_1}^{l_{1,k}}}\otimes \cdots \otimes
\q_{b_{\alpha_n}^{l_{n,k}}},
\]for all $k= 1, \ldots, r$, so that $\mbox{co-rank~} \q(\bm{\alpha}; \tilde{A})
\leq r$. The reverse inequality will follow, by
virtue of Lemma \ref{lemma 2}, if we can construct a closed subspace
$\cle\subseteq \q(\bm{\alpha}; \tilde{A})$ such that
$\q(\bm{\alpha}; \tilde{A})\ominus \cle$ is a quotient module of
$H^2(\D^n)$ and $\mbox{rank}~\cle=r$ for
$(P_{\cle}M_{z_1}^*|_{\cle},\dots,P_{\cle}M_{z_n}^*|_{\cle})$. To
this end, let \[ g_k:=
b_{\alpha_1}^{l_{1,k}-1}M_{z}^*b_{\alpha_1}\ot \cdots\ot
b_{\alpha_n}^{l_{n,k}-1}M_{z}^*b_{\alpha_n} \in
\q_{b_{\alpha_1}^{l_{1,k}}}\otimes \cdots \otimes
\q_{b_{\alpha_n}^{l_{n,k}}},
\]
for all $k = 1, \ldots,
r$. By virtue of \eqref{equivalent definition} we conclude that
$\{g_k\}_{k=1}^r$ is an orthogonal set, and hence
$\cle:=\bigoplus_{k=1}^r \Comp g_k$ is an $r$ dimensional subspace
of $\q(\bm{\alpha}, \tilde{A})$ and $\q(\bm{\alpha};
\tilde{A})\ominus \cle$ is a quotient module of $H^2(\D^n)$. Now
from (\ref{norm-1}) it follows that
\[\|g_k\|^2 = \prod_{j=1}^n (1 - |\alpha_j|^2),\]and
by \eqref{equivalent definition} we have $ \langle g_{k'},
M_{z_i}^*g_k\rangle=0$ for any $1\le k < k'\le r$ and $1\le i\le n$.
Thus using (\ref{P-1}) one can have
\[
P_{\cle}M_{z_i}^*g_k=P_{\Comp g_k}M_{z_i}^*g_k=\bar{\alpha_i}
 g_k,
\]for all $i=1,\dots,n$, and $k=1,\dots,m$.
This implies that \[P_{\cle}M_{z_i}^*|_{\cle}=
\bar{\alpha_i}I_{\cle}. \quad \quad (i=1,\dots,n)\] Since
$\mbox{dim~}\cle = r$, we see that $\mbox{rank}~\cle=r$ for
$(P_{\cle}M_{z_1}^*|_{\cle},\dots,P_{\cle}M_{z_n}^*|_{\cle})$.
 This completes the proof.
\end{proof}

\newsection{Minimal representations of quotient modules}

Let $\vp$ be a Blaschke product and $\xi$ be a non-constant factor
of $\vp$. The \textit{order} of $\xi$ in $\varphi$, denoted by
$\mbox{ord} (\vp, \xi)$, is the unique integer $m$ such that $\vp =
\xi^m \psi$ for
 an inner function $\psi$ and $\xi \not| \psi$. In particular if
$b_{\alpha}$ is a prime factor of $\vp$, then $\mbox{ord}(\vp,
b_{\alpha})$ denotes the zero order of $\vp$ at $\alpha$.

For the rest of this paper, we fix $\bm{\Phi} = (\Phi_1, \ldots,
\Phi_n)$, where $\Phi_i = \{\varphi_{i,k}\}_{k=-\infty}^{\infty}$ is
a sequence of Blaschke products with a least common multiple
$\varphi_i$, $i = 1, \ldots, n$. Our main concern here is to analyze
and compute the co-rank of the following Rudin's quotient module
\begin{equation}\label{rudin} \q_{\bm{\Phi}} =
\bigvee_{k=-\infty}^{\infty}\q_{\varphi_{1,k}} \otimes \cdots
\otimes \q_{\varphi_{n,k}}.
\end{equation} By defining
\begin{equation}\label{index set} \Lambda_k:=\I_{\varphi_{1,k}}
\times\dots\times\I_{\varphi_{n,k}}\quad (k\in\Int)\quad \text{and }
\Lambda: =\bigcup_{k\in\Int}\Lambda_k,
\end{equation}Corollary \ref{maincor3} shows that
\[
\q_{\bm{\Phi}} = \bigvee_{ (\xi_1,\dots,\xi_n) \in\Lambda}
\q_{\xi_{1}}\otimes \cdots \otimes \q_{\xi_{n}}.
\]
Now let $(\xi_1,\dots,\xi_n)\in \Lambda_k$ and $k\in\Int$. Then
$\xi_i=P(\xi_i)^{l_{i,k}}$, where $P(\xi_i)$ is the prime inner
function corresponding to $\xi_i$ and
\begin{equation}
\label{order} l_{i,k}= \mbox{ord}(\xi_i,
P(\xi_i))=\mbox{ord}(\varphi_{i,k}, P(\xi_i)). \quad\quad
(i=1,\dots,n)
\end{equation}

\noindent Thus $(\xi_1,\dots,\xi_n)\in \Lambda_k$ corresponds
precisely to a tuple of prime inner functions
$(P(\xi_1),\dots,P(\xi_n))$ and a tuple of natural numbers
$(l_{1,k}, \ldots, l_{n,k}) \in \mathbb{N}^n$. Moreover,
\begin{equation}\label{ind-set2} \clq_{\bm{\Phi}} =
\bigvee_{(\xi_1,\dots,\xi_n) \in \Lambda} \q_{P(\xi_{1})^{l_{1,k}}}
\otimes \cdots \otimes \q_{P(\xi_{n})^{l_{n,k}}}.\end{equation}
Also, note that for each $i = 1, \ldots, n$, $P(\xi_i) =
b_{\alpha_i}$ for some $\alpha_i \in \D$. Based on this observation,
we define the \textit{zero set} of the tuple $(\xi_1,\dots,\xi_n)
\in \Lambda_m$ as follows:
\[
Z(\xi_1,\dots,\xi_n) = \{k\in \mathbb{Z}: P(\xi_{i})| \varphi_{i,k}
~\text{for all}~i=1,2,\ldots,n\}.
\]
Note that $Z(\xi_1,\dots,\xi_n)$ is a countable and non-empty set
(since $m \in Z(\xi_1,\dots,\xi_n)$).

If we define the quotient module $\q(\xi_1,\dots,\xi_n)$ by

\begin{equation}\label{ntilde}
\q(\xi_1,\dots,\xi_n): = \bigvee_{k\in Z(\xi_1,\dots,\xi_n)}
\q_{P(\xi_{1})^{l_{1,k}}} \otimes \cdots \otimes
\q_{P(\xi_{n})^{l_{n,k}}},
\end{equation}
then by (\ref{ind-set2}) it follows that
\[
\q_{\bm{\Phi}} =\bigvee_{(\xi_1,\dots,\xi_n)\in\Lambda}
\q(\xi_1,\dots,\xi_n).
\]

This sets the stage for the following result concerning a minimal
representation of $\q(\xi_1,\dots,\xi_n)$:

\begin{Proposition}
\label{minimal of q} Let $\q(\xi_1,\dots,\xi_n)$ be as in
(\ref{ntilde}) for some $(\xi_1,\dots,\xi_n)\in\Lambda$. Then there
exists a finite subset $\tilde{Z}(\xi_1,\dots,\xi_n)$ of
$Z(\xi_1,\dots,\xi_n)$ with minimal cardinality such that
\begin{equation}\label{ztilde}
\q(\xi_1,\dots,\xi_n)= \bigvee_{k\in \tilde{Z}(\xi_1,\dots,\xi_n)}
\q_{P(\xi_{1})^{l_{1,k}}}\otimes \cdots \otimes
\q_{P(\xi_{n})^{l_{n,k}}}.
\end{equation}
\end{Proposition}
\begin{proof}
First consider the set of tuples $\{(l_{1,k}, \ldots,l_{n,k})\in
\Nat^n : k\in Z(\xi_1,\dots,\xi_n) \}$, where $l_{i,k}$ is defined as
in ~\eqref{order} for $i = 1, \ldots, n$. Then construct
$\tilde{Z}(\xi_1,\dots,\xi_n)\subseteq Z(\xi_1,\dots,\xi_n)$ by
removing those $k \in Z(\xi_1,\dots,\xi_n)$ for which there exists
$k'\in Z(\xi_1,\dots,\xi_n)$ such that $l_{i,k'}\ge l_{i,k}$ for all
$i=1,\dots,n$. Then the equality \eqref{ztilde}, for
$\tilde{Z}(\xi_1,\dots,\xi_n)$ as constructed above, follows from
\eqref{inc-red} and \eqref{inc-red2}. Finally, since the sequence
$\{\varphi_{i,k}\}_{k=-\infty}^{\infty}$ has a least common
multiple, we obviously have
\[\sup_{k\in\Int}l_{i,k}=
\sup_{k\in\Int}\mbox{ord}(\varphi_{i,k}, P(\xi_i))<\infty, \quad
\quad\quad (i=1,\dots,n)\]and hence it follows that the cardinality
of $\{(l_{1,k}, \ldots,l_{n,k})\in \Nat^n : k\in
Z(\xi_1,\dots,\xi_n) \}$ is finite. Therefore
$\tilde{Z}(\xi_1,\dots,\xi_n)$ is a finite set. This concludes the
proof of the proposition.
\end{proof}

We will call the representation in \eqref{ztilde} the
\textit{minimal representation} of $\q(\xi_1,\dots,\xi_n)$.

The following result is useful in connection with the existence of
minimal index set $\tilde{Z}(\xi_1,\dots,\xi_n)$.

\begin{Proposition}\label{key lemma}
Let $\{\varphi_{i,k}\}_{k=-\infty}^{\infty}$ be a sequence of
Blaschke products with a least common multiple inner function
$\varphi_i$, for all $i=1,\dots,n$. Then for each
$(\xi_1,\dots,\xi_n)\in\Lambda$,
\[
\mbox{co-rank}~\q(\xi_1,\dots,\xi_n)=
\#\tilde{Z}(\xi_1,\dots,\xi_n),\] where
$\tilde{Z}(\xi_1,\dots,\xi_n)$ is the minimal index set for
$\q(\xi_1,\dots,\xi_n)$ as in Proposition~\ref{minimal of q}.
\end{Proposition}
\begin{proof}
Let us set, for $i = 1, \ldots, n$, \[P(\xi_i)=b_{\alpha_i},\]for
some $\alpha_i\in\D$, and $\bm{\alpha} = (\alpha_1, \ldots,
\alpha_n)$. Using the notation in \eqref{nset}, we have
\[\q(\xi_1,\dots,\xi_n) = \q(\bm{\alpha}; A(\xi_1,\dots,\xi_n)) =
\q(\bm{\alpha}; \tilde{A}(\xi_1,\dots,\xi_n)),\] where
\[A(\xi_1,\dots,\xi_n)=\{(l_{1,k}, \ldots,l_{n,k})\in \Nat^n
: k\in Z(\xi_1,\dots,\xi_n)\},\] and \[\tilde{A}(\xi_1,\dots,\xi_n)=
\{(l_{1,k},\ldots,l_{n,k})\in \Nat^n : k\in
\tilde{Z}(\xi_1,\dots,\xi_n)\},\] and  $\q(\bm{\alpha};
\tilde{A}(\xi_1,\dots,\xi_n))$ is the minimal representation of
$\q(\bm{\alpha};A(\xi_1,\dots,\xi_n))$. Then the desired equality
follows from Proposition ~\ref{mainlemma4}. This completes the
proof.
\end{proof}

Now we observe that for $(\xi_1,\dots,\xi_n),
(\xi_1',\dots,\xi_n')\in \Lambda$, if $P(\xi_i)=P(\xi_i')$ for all
$i=1,\dots,n$, then $\q(\xi_1,\dots,\xi_n)=\q(\xi_1',\dots,\xi_n')$.
Consequently, $\sim$ is an equivalence relation on $\Lambda$, where
$(\xi_1,\dots,\xi_n)\sim (\xi_1',\dots,\xi_n')$ if
$P(\xi_i)=P(\xi_i')$ for all $i=1,\dots,n$. This readily implies
that
\[
\q_{\bm{\Phi}} = \bigvee_{ (\xi_1,\dots,\xi_n)\in [\Lambda]}
\q(\xi_1,\dots,\xi_n),
\]
where $[\Lambda]:=\Lambda/\sim$ is the set of all equivalence
classes in $\Lambda$.

\newsection{Co-rank of $\q_{\bm{\Phi}}$}

In this section we compute the co-rank of the quotient module of the
form \eqref{rudin}.

\begin{Theorem}\label{mainth1}
Let $\{\varphi_{i,k}\}_{k=-\infty}^{\infty}$ be a sequence of
Blaschke products with a least common multiple inner function
$\varphi_i$, $i=1,\dots,n$, and let
\[
 \q_{\bm{\Phi}} = \bigvee_{k=-\infty}^{\infty}\q_{\varphi_{1,k}}
\otimes \cdots \otimes \q_{\varphi_{n,k}}.
\]
Then
\begin{align*}
\mbox{co-rank~} \q_{\bm{\Phi}} = \sup_{(\xi_1,\dots,\xi_n)\in
\Lambda} \mbox{co-rank~} \q(\xi_1,\dots,\xi_n) =
\sup_{(\xi_1,\dots,\xi_n)\in \Lambda}\#\tilde{Z}(\xi_1,\dots,\xi_n),
\end{align*}
where $\Lambda$ is as in ~\eqref{index set} and for
$(\xi_1,\dots,\xi_n)\in \Lambda$, $\tilde{Z}(\xi_1,\dots,\xi_n)$ is
the minimal index set for the minimal representation of
$\q(\xi_1,\dots,\xi_n)$ as in ~\eqref{ztilde}.
\end{Theorem}

\begin{proof}
By Proposition \ref{key lemma}, we have
\[
\sup_{(\xi_1,\dots,\xi_n)\in \Lambda}
\mbox{co-rank~}\q(\xi_1,\dots,\xi_n) = \sup_{(\xi_1,\dots,\xi_n)\in
\Lambda}\#\tilde{Z}(\xi_1,\dots,\xi_n).
\]
Now to see the first equality, let $(\xi_1,\dots,\xi_n)\in\Lambda$.
Set
\[a_i:=\sup\{l_{i,m}: m\in \tilde{Z}(\xi_1,\dots,\xi_n)\},\]
where $l_{i,m} = \mbox{order~} (\varphi_{i,m}, P(\xi_i))$, $m\in
\tilde{Z}(\xi_1,\dots,\xi_n)$, and $1\le i\le n$. Since
$\{\varphi_{i,k}\}_{k=-\infty}^\infty$ has a least common multiple,
then  $a_i<\infty$, and
\[
\varphi_{i}(\xi_i) : = \frac{\varphi_i}{P(\xi_i)^{a_i}},
\]is a Blaschke product for all $i = 1, \ldots, n$.
Since $\varphi_{i}(\xi_i)$ and $P(\xi_{i})^{t}$ are relatively prime
for any $t \in\Nat\setminus \{0\}$ and $i=1,2,\ldots,n$, by
Corollary~\ref{mainlemma1} we conclude that
\begin{equation}\label{main5} \begin{split}
M_{\varphi_{1}(\xi_1)}^*\ot \cdots\ot M_{\varphi_{n}(\xi_n)}^*
\left(\q_{P(\xi_{1})^{l_{1,m}}}\otimes \cdots \otimes
\q_{P(\xi_{n})^{l_{n,m}}}\right) = \q_{P(\xi_{1})^{l_{1,m}}}\otimes
\cdots \otimes \q_{P(\xi_{n})^{l_{n,m}}},
\end{split}
\end{equation}
for all $m\in Z(\xi_1,\dots,\xi_n)$. On the other hand, let
$(\xi_1',\dots,\xi_n')\in \Lambda$ be such that $P(\xi_i)$ is not a
factor of $\xi_i'$ for some $1\le i\le n$. This implies that
$\xi'_i|\varphi_{i}(\xi_i)$ for some $1\le i \le n$.
Consequently, by Lemma~\ref{lemma 1}
\begin{equation}\label{main4}
M_{\varphi_1(\xi_1)}^*\ot \cdots\ot M_{\varphi_{n}(\xi_n)}^*
\left(\q_{\xi_{1}'}\otimes \cdots \otimes \q_{\xi_{n}'}\right)\\ =
\{0\}.
\end{equation} Thus combining \eqref{main4} and \eqref{main5},
we have
\[
M_{\varphi_{1}(\xi_1)}^* \ot\cdots\ot
M_{\varphi_{n}(\xi_n)}^*\left(\q(\xi_1,\dots,\xi_n)\right) =
\q(\xi_1,\dots,\xi_n),
\]
and
\[
M_{\varphi_{1}(\xi_1)}^* \ot\cdots\ot
M_{\varphi_{n}(\xi_n)}^*\left(\q_{\bm{\Phi}}\right) =
\q(\xi_1,\dots,\xi_n).
\]
This yields $\mbox{co-rank}~\q(\xi_1,\dots,\xi_n) \leq
\mbox{co-rank}~\q_{\bm{\Phi}}$ for all $(\xi_1,\dots,\xi_n)\in
\Lambda$.

\noindent To prove the reverse inequality, we may assume that
\[ m_0: =\sup_{(\xi_1,\dots,\xi_n)\in
\Lambda}\#\tilde{Z}(\xi_1,\dots,\xi_n)<\infty.
\]
It is thus enough to show that $\q_{\bm{\Phi}}$ is (co-)generated by
$m_0$ vectors. We proceed next with the detailed construction of a
co-generating set of vectors of cardinality $m_0$.

\noindent Let $k\in\Int$ and $(\xi_1,\dots,\xi_n)\in \Lambda_k$.
Also for all $m\in Z(\xi_1,\dots,\xi_n)$, let $f_m(\xi_i) \in
\q_{P(\xi_i)^{l_{i,m}}}$ be a unit star-cyclic vector of
$\q_{P(\xi_i)^{l_{i,m}}}$, $i=1,\dots,n$. Obviously
$f_m(\xi_1)\otimes \cdots\otimes f_m(\xi_n)$ is a star-cyclic vector
of $\q_{P(\xi_1)^{l_{1,m}}}\otimes \cdots \otimes
\q_{P(\xi_{n})^{l_{n,m}}}$ for all $m\in Z(\xi_1,\dots,\xi_n)$. In
this setting, we relabel the set of unit vectors $\{f_m(\xi_1)\otimes
\cdots\otimes f_m(\xi_n): m\in \tilde{Z}(\xi_1,\dots,\xi_n)\}$ by
defining a bijective function
\[g:\{1,\dots,\#
\tilde{Z}(\xi_1,\dots,\xi_n)\}\to \tilde{Z}(\xi_1,\dots,\xi_n),\] and
letting
\begin{equation*}
F_r(\xi_1,\dots,\xi_n) =
\begin{cases}
f_{g(r)}(\xi_1) \otimes \cdots \otimes
f_{g(r)}(\xi_n)\hspace{01.8cm} \mbox{if~}
1\le r\le \#\tilde{Z}(\xi_1,\dots,\xi_n); &\\
0\hspace{6cm}\mbox{if~} \#\tilde{Z}(\xi_1,\dots,\xi_n)< r\leq m_0.
\end{cases}
\end{equation*}
Thus corresponding to each $(\xi_1,\dots,\xi_n)\in [\Lambda]$, we
have $m_0$ number of vectors of the above form. We now use these
facts to define
\[
G_r = \sum_{(\xi_1,\dots,\xi_n)\in [\Lambda]}
C(\xi_1,\dots,\xi_n)F_r(\xi_1,\dots,\xi_n) ,\hspace{0.5cm} 1\leq
r\leq m_0,
\]
where the sum is over a countable set and the constants
$C(\xi_1,\dots,\xi_n)$ are so that the above sum converges. Then
$G_r\in \q_{\bm{\Phi}}$ for $1\leq r\leq m_0$. Next consider the
subspace
\[
\Omega = \bigvee_{ t_1,t_2,\ldots,t_{n}\in \Nat}
M_{z_1}^{*{t_1}}\ot \cdots\ot M_{z_n}^{*{t_n}} \{G_1,
\ldots,G_{m_0}\}.
\]
Since $G_r\in \q_{\bm{\Phi}}$ for $1\leq r\leq m_0$, we obviously
have $\Omega \subseteq \q_{\bm{\Phi}}$. Now for
$(\xi_1,\dots,\xi_n)\in\Lambda$ and $1\leq r\leq
\tilde{Z}(\xi_1,\dots,\xi_n)$, we have
\[
M_{\varphi_{1}(\xi_1)}^* \ot\cdots\ot
M_{\varphi_{n}(\xi_n)}^*\left(G_r\right)\in \Omega,
\]
and using ~\eqref{main4} and \eqref{main5} we conclude that
\begin{align}
\label{mainlast} &M_{\varphi_{1}(\xi_1)}^*\ot \cdots\ot
M_{\varphi_{n}(\xi_n)}^*\left(G_r\right)\\
& = C(\xi_1,\dots,\xi_n)M_{\varphi_{1}(\xi_1)}^* \ot\cdots\ot
M_{\varphi_{n}(\xi_n)}^*\left(F_r(\xi_1,\dots,\xi_n)\right)\nonumber\\
& =
C(\xi_1,\dots,\xi_n)M_{\varphi_{1}(\xi_1)}^*\left(f_{g(r)}(\xi_1)\right)
\otimes \cdots \otimes
M_{\varphi_{n}(\xi_n)}^*\left(f_{g(r)}(\xi_n)\right)\nonumber.
\end{align}
By virtue of Corollary~\ref{mainlemma1},
\[
M_{\varphi_{1}(\xi_1)}^*\left(f_{g(r)}(\xi_1)\right) \otimes \cdots
\otimes M_{\varphi_{n}(\xi_n)}^*\left(f_{g(r)}(\xi_n)\right)
\]
is a star-cyclic vector of
\[
\q_{P(\xi_{1})^{l_{1,g(r)}}}\otimes \cdots \otimes
\q_{P(\xi_{n})^{l_{n,g(r)}}}.
\]
Hence we obtain $\q(\xi_1,\dots,\xi_n) \subseteq \Omega$ for all
$(\xi_1,\dots,\xi_n)\in[\Lambda]$, and consequently $\Omega =
\q_{\bm{\Phi}}$. As a result we have
$\mbox{co-rank}~\q_{\bm{\Phi}}\leq m_0$, and this concludes the
proof.
\end{proof}

Let $A \subsetneqq \{1, \ldots, n\}$ and $\Phi_i =
\{\varphi_{i,k}\}_{k=-\infty}^\infty$ be a sequence of Blaschke
products with no common non-constant inner function, $i = 1, \ldots,
n$. The contents of the last section can be adopted to a general
class of Rudin's quotient modules $\q_{\bm{\Phi}}$, where $\Phi_i$
is increasing for $i \in A$ and decreasing for $i \in B: = \{1,
\ldots, n\} \setminus A$.

\noindent In this case for each $(\xi_1,\dots,\xi_n)\in\Lambda$,
\[Z(\xi_1,\dots,\xi_n) = \{k\in\Int: r_1\le k\le r_2\},\]
where
\begin{align}
\label{r}
& r_1=\min\{k\in\Int: P(\xi_i)|\varphi_{i,k} \text{ for all } i\in
A\},\text{ and }\\
& r_2=\max\{k\in\Int: P(\xi_i)|\varphi_{i,k} \text{ for
all } i\in B \}.
\nonumber
\end{align}
Note first that $|r_1|, |r_2| < \infty$. This follows from the fact
that $\Phi_i$, $i = 1, \ldots, n$, does not have any common inner
factor. Consequently, $Z(\xi_1,\dots,\xi_n)$ is a finite set. Note
that in the proof of Theorem~\ref{mainth1}, the assumption that each
of the sequence has least common multiple has been used to ensure
that $\#\tilde{Z}(\xi_1,\dots,\xi_n)<\infty$ and also used to
construct inner functions so that ~\eqref{main5}, ~\eqref{main4} and
~\eqref{mainlast} holds. In the present consideration, we can still
do this by defining
\begin{equation}
\varphi_i(\xi_i)=\left\{\begin{array}{cl}
\frac{\varphi_{i,r_1}}{P(\xi_i)^{l_{i,r_1}}}& \text{if } i\in B,\\
\frac{\varphi_{i,r_2}}{P(\xi_i)^{l_{i,r_2}}}& \text{if } i\in A,
\end{array}\right.
\end{equation}
where $i = 1, \ldots, n$, and $r_1$ and $r_2$ are as in ~\eqref{r}. We
can see now the proof of the co-rank equality, as in
Theorem~\ref{mainth1}, for this quotient module follows along the
same line as the proof of Theorem~\ref{mainth1}. Therefore, we have
the following theorem:

\begin{Theorem}\label{mainth2}
Let $A$ be a proper non-empty subset of $\{1, \ldots, n\}$ and $B: =
\{1, \ldots, n\} \setminus A$, and let
$\Phi_i=\{\varphi_{i,k}\}_{k=-\infty}^\infty$ be a sequence of
Blaschke products with no common non-constant inner function, $i=1,\dots,n$. Also
let $\Phi_i$ be increasing for all $i\in A$ and decreasing for all
$i\in B$. Then
\begin{equation*}
\mbox{co-rank}~\q_{\bm{\Phi}}=\sup_{(\xi_1,\dots,\xi_n)\in
\Lambda}\#\tilde{Z}(\xi_1,\dots,\xi_n).
\end{equation*}
\end{Theorem}

\begin{Remark}\label{ce}
The above theorem, restricted to $n=2$ case, is related to Theorem
4.2 in ~\cite{I1}. However, the formulation of Theorem 4.2 in
\cite{I1} turns out to be incorrect. This will be discussed at the
end of the final section.
\end{Remark}

In the present context, for $(\xi_1,\dots,\xi_n)\in \Lambda$, it is
also possible to describe the set $\tilde{Z}(\xi_1,\dots,\xi_n)$.
Let $l_{i,k} = \mbox{order} (\vp_{i,k}, P(\xi_i))$, as before (see
\eqref{order}), for all $i = 1, \ldots, n$ and $k \in Z(\xi_1,\dots,\xi_n)$. Note
that $l_{i,k}\ge l_{i,k+1}$ for all $i\in B$, and $l_{i,k}\le
l_{i,k+1}$ for all $i\in A$. Now we proceed to construct
$\tilde{Z}(\xi_1,\dots,\xi_n)$ as follows. Set
\begin{equation}
\label{zeta} \zeta_{i,k}:=  \frac{\varphi_{i,k}}{\varphi_{i,k-1}},
\quad \quad (i\in A, k\in Z(\xi_1,\dots,\xi_n))
\end{equation}
and
\begin{equation}
\label{Iset}
 I(\xi_1,\dots,\xi_n):=\{ k\in Z(\xi_1,\dots,\xi_n): P(\xi_i)|\zeta_{i,k}
 \text{ for some } i\in A\}.
\end{equation}
It is clear that $r_1\in I(\xi_1,\dots,\xi_n)$ and hence
$\tilde{Z}(\xi_1,\dots,\xi_n)=\{r_1\}$ when $\#
I(\xi_1,\dots,\xi_n)=1$. Now suppose we have $\#
I(\xi_1,\dots,\xi_n) = m+1 > 1$, for some $m \in \mathbb{N}$, and
(without loss of any generality)
\[I(\xi_1,\dots,\xi_n)= \{k_0=r_1<k_1<k_2<\dots<k_m\le r_2\}.\]
Define
\begin{equation}
\label{eta} \eta_{i,k_j}:= \frac{\varphi_{i,
k_j}}{\varphi_{i,k_{j+1}}}.\quad \quad (0\le j\le m-1, i\in B)
\end{equation}
Then $\tilde{Z}(\xi_1,\dots,\xi_n)=\{k_m\}\cup\{k_j\in
I(\xi_1,\dots,\xi_n): k_j\ne k_m, P(\xi_i)|\eta_{i, k_j} \text { for
some } i\in B\}$.

The above discussion, along with Theorem \ref{mainth2}, may be
summarized in the following.

\begin{Theorem}
\label{zt} Let $\Phi_i=\{\varphi_{i,k}\}_{k=-\infty}^\infty$,
$i=1,\dots,n$, be as in the statement of Theorem~\ref{mainth2}. Then
\begin{equation*}
\mbox{co-rank}~\q_{\bm{\Phi}}=\sup_{(\xi_1,\dots,\xi_n)\in
\Lambda}\#\tilde{Z}(\xi_1,\dots,\xi_n).
\end{equation*}
Moreover, for all $(\xi_1, \ldots, \xi_n) \in \Lambda$,
\[\#\tilde{Z}(\xi_1,\dots,\xi_n)=1+ \#\{ k_j\in
I(\xi_1,\dots,\xi_n): k_j\ne k_m,  P(\xi_i)|\eta_{i,k_j} \text{ for
some } i\in B \},\]
where $I(\xi_1,\dots,\xi_n)$ is as in~\eqref{Iset}, and $\eta_{i,k_j}$
is as in ~\eqref{eta}.
\end{Theorem}

\newsection{Concluding Remarks}

We now present a simple example which illustrate the main idea of
this paper.

\noindent Let
$\{\{\alpha_{i,k}\}_{k=-\infty}^{\infty}:i=1,\dots,n\}$ be a
collection of sequences of distinct points in $\D$ such
that\[\sum_{k=-\infty}^{\infty}(1-|\alpha_{i,k}|)<\infty. \quad
\quad (i=1,\dots,n)\] Let $A$ be a proper non-empty subset of
$\{1,\dots,n\}$ and $B:=\{1,\dots,n\}\setminus A$. Also consider the
following sequences of Blaschke products
\begin{equation*} \varphi_{i,k}=\left\{\begin{array}{cl}
\prod_{j=k}^{\infty} b_{\alpha_{i,j}}& \text{ if } i\in B;\\
\prod_{j=-\infty}^k b_{\alpha_{i,j}} & \text{ if } i\in A.
\end{array}\right.
\end{equation*}
Consequently, $\{\varphi_{i,k}\}_{k=-\infty}^{\infty}$ is an
increasing sequence for each $i\in
 A$ and decreasing sequence for each $i\in B$. Let
$(\xi_1,\dots,\xi_n)\in \Lambda_m$ for some $m\in\Int$. Then
$\xi_i=b_{\alpha_{i,k_i}}$, where $k_i\ge m$ for all $i\in B$ and
$k_i\le m$ for all $i\in A$. In this case, $r_1=\max\{k_i: i\in A\}$
and $r_2=\min\{k_i:i\in B\}$. From the fact that the set of points
are distinct, we deduce that $I(\xi_1,\dots,\xi_n)=\{r_1\}$. Hence
$\tilde{Z}(\xi_1,\dots,\xi_n)=\{r_1\}$, and consequently,
$\mbox{co-rank~} \clq_{\bm{\Phi}} = 1$.

To end this paper, we construct a counter example, as promised in
Remark \ref{ce}, to point out an error in the formulation of Theorem
4.2 in ~\cite{I1}.

Let $\Phi=\{\varphi_m\}_{m=-\infty}^{\infty}$ be a decreasing
sequence of Blaschke products, and let
$\Psi=\{\psi_m\}_{m=-\infty}^{\infty}$ be an increasing sequence of
Blaschke products such that each of the sequence does not have any
non-constant common inner factor. First, for the sake of convenience
we state Theorem 4.2 from \cite{I1}.

\begin{Theorem}[Theorem 4.2, \cite{I1}]\label{error}
Let $\Phi$ and $\Psi$ be as above. Then
\[
\mbox{co-rank}~ \clq_{(\Phi,\Psi)} = \sup_{j\ge
1}\#\{n:\zeta_n(\alpha_j)=\xi_n(\beta_j)=0, -\infty<n<\infty\},
\]
where
\[
\zeta_m=\varphi_m/\varphi_{m+1} \text{ and } \xi_m
=\psi_m/\psi_{m-1}, \quad (m \in\Int)
\]
and $(\alpha_j,\beta_j)_{j\ge 1}$ is the enumeration of the countable set
$Z=\{(\alpha,\beta)\in\D^2:\varphi_m(\alpha)=\psi_m(\beta)=0\
\text{for some } m\in\Int\}$.
\end{Theorem}
We need some more notations in the spirit of \cite{I1}. For $j\ge 1$
and $(\alpha_j,\beta_j)\in Z$, define
\[
 Z_j=\{n:\varphi_n(\alpha_j)=\psi_n(\beta_j)=0, -\infty<n<\infty\},
\]
and
\[
\mathcal{N}_{j}=\sum_{n\in Z_j}\clq_{b_{\alpha_j}^{r_{j,n}}}\ot
\clq_{b_{\beta_j}^{s_{j,n}}},
\]
where $r_{j,n}=\mbox{ord}(\varphi_n,b_{\alpha_j})$ and
$s_{j,n}=\mbox{ord}(\psi_n,b_{\beta_j})$ for all $n\in Z_j$.

We note first that the proof of Theorem~\ref{error} (or Theorem 4.2
in \cite{I1}), as pointed out by the authors, is based on the
following identity:
\begin{equation}
\label{error identity}
 \#\{n:\zeta_n(\alpha_j)=\xi_n(\beta_j)=0, -\infty<n<\infty\}=m_{j},
\end{equation}
where $j\ge 1$, $(\alpha_j,\beta_j)\in Z$ and $m_j$ is the minimum
number required to represent the quotient module $\mathcal{N}_j$.
However,  the above equalities does not hold in general, and hence
Theorem~\ref{error} is also incorrect. The next example demonstrates
that the above equality and Theorem \ref{error} are incorrect.

Let $\{a_m\}_{m=-\infty}^{\infty}$ and
$\{c_m\}_{m = -\infty}^{\infty}$ be a pair of sequences of
points in $\D$ such that
\[\sum_{m=-\infty}^\infty (1 - |a_m|), \sum_{m=-\infty}^\infty (1 - |c_m|) < \infty,\]
and all elements are distinct but $a_k=a=a_{k+3}$ and
$c_k=c=c_{k+2}$ for some fixed $k\in\Int$. Consider the following
sequences of Blaschke products $\Phi=\{\varphi_m\}_{m\in\Int}$ and
$\Psi=\{\psi_m\}_{m\in\Int}$, where
\[
 \varphi_m:=\prod_{j=m}^{\infty}b_{a_j} \text{ and }
 \psi_m:=\prod_{j=-\infty}^m b_{c_j}. \quad \quad (m\in\Int)
 \]
Notice that
\[
 \zeta_m=\varphi_m/\varphi_{m+1}=b_{a_m} \text{ and }
 \xi_m =\psi_m/\psi_{m-1}=b_{c_m}. \quad (m \in\Int)
\]
Furthermore we notice that if $\alpha_j=a$ and $\beta_j=c$, then
\[
 \#\{m:\zeta_m(a)=\xi_m(c)=0, m\in\Int\}=
 \#\{m: b_{a}|b_{a_m}, b_{c}|b_{c_m}, m\in\Int\}=\#\{k\}= 1,
\]
whereas \[Z_j=\{k, k+1, k+2, k+3\},\] and
\begin{equation}
\label{N_j}
 N_j=\clq_{{b_a}^2}\ot\clq_{b_c}\vee \clq_{b_a}\ot\clq_{b_c}\vee
  \clq_{b_a}\ot\clq_{{b_c}^2}\vee \clq_{b_a}\ot\clq_{{b_c}^2}=
  \clq_{{b_a}^2}\ot\clq_{b_c}\vee\clq_{b_a}\ot\clq_{{b_c}^2}.
\end{equation}
The above identity shows that $m_j$ has to be $2$, and hence ~\eqref{error identity}
is not correct.

Also note that for any $(\alpha_j, \beta_{j})\in Z $,
\[
 \#\{m:\zeta_m(\alpha_j)=\xi_m(\beta_j)=0, m\in\Int\}\le 1.
\]
Therefore, by Theorem~\ref{error}, the co-rank of the
quotient module $\clq_{(\Phi,\Psi)}$ is $1$.

\noindent However, since the the co-rank of $\mathcal{N}_j=2$
(follows from \eqref{N_j}), by Theorem~\ref{mainth1} (or by Theorem
4.1 in ~\cite{I1}) the co-rank of $\clq_{(\Phi,\Psi)}$ is at least
$2$. This shows that the
 formulation of Theorem 4.2 in  \cite{I1} is also incorrect.
% This error is actually due to the identity claimed just before
% Theorem 4.2 in~\cite{I1}
% (see ~\cite{I1}, page 1466, the paragraph before Theorem 4.2).

On the other hand, one can easily calculate, using the formula in
Theorem~\ref{zt}, that  \[\#\tilde{Z}(b_a,b_c)=2,\] and for any
other $(b_{a_i}, b_{c_j})\in\Lambda$
\[\#\tilde{Z}(b_{a_i},b_{c_j})\le 1.\] Consequently,
the co-rank of $\clq_{(\Phi,\Psi)}$ is precisely $2$.

%In particular, let $\alpha_j:=a_2$ and $\beta_j:= c_1$. Then for the
%pair $(\alpha_j,\beta_j)$,
%\[
% \#\{m:\zeta_m(\alpha_j)=\xi_m(\beta_j)=0, m\in\Int\}=
% \#\{m: b_{a_2}|b_{a_m}, b_{c_1}|b_{c_m}, m\in\Int\}=0,
%\]
%where we use the fact that the sequence of points chosen are all
%distinct. Therefore by the identity before Theorem 4.2 in
%~\cite{I1}, $m_j=0$, where $m_j$ is the minimum number required for
%the representation of the finite dimensional quotient module
%$\mathcal{N}_j$ which, in this particular case, is given by
%$\q_{b_{a_2}}\ot\q_{b_{c_1}}$. Since the number $m_j$ can not be
%$0$,

%In other words, the
%formulation of Theorem 4.2 in  \cite{I1} is also incorrect.

\vskip10pt

\noindent\textbf{Acknowledgement:} We would like to thank the
referee for providing us with constructive comments and suggestions.
The first author acknowledges with thanks financial support  from
the Department of Atomic Energy, India through N.B.H.M Post Doctoral
Fellowship and the second author is grateful to Indian Statistical
Institute, Bangalore Center for warm hospitality.

%\begin{thebibliography}{ABC}

\end{document}